\documentclass[11pt]{article}
\usepackage{etex}
\reserveinserts{28}
\usepackage[all]{xy}

\usepackage{amsfonts}
\usepackage{amsthm}
\usepackage{enumerate}
\usepackage{graphicx}
\usepackage{mathrsfs}
\usepackage{bm}
\usepackage{cite}
\usepackage{amssymb,amsmath} 
\usepackage{makecell}
\usepackage{tikz}
\usepackage{pgf}
\usepackage{tikz}
\usetikzlibrary{patterns}
\usepackage{pgffor}
\usepackage{pgfcalendar}
\usepackage{pgfpages}
\usepackage{shuffle,yfonts}
\usepackage{mathtools}
\DeclareFontFamily{U}{shuffle}{}
\DeclareFontShape{U}{shuffle}{m}{n}{ <-8>shuffle7 <8->shuffle10}{}

\newcommand{\tn}{{\tilde{n}}}

\DeclareMathOperator{\sn}{{\rm sn}}

\DeclareMathOperator{\cn}{{\rm cn}}
\DeclareMathOperator{\dn}{{\rm dn}}

\DeclareMathOperator\Res{{\rm Res}}

 \allowdisplaybreaks
\usetikzlibrary{arrows,shapes,chains}
 \allowdisplaybreaks

\catcode`!=11
\let\!int\int \def\int{\displaystyle\!int}
\let\!lim\lim \def\lim{\displaystyle\!lim}
\let\!sum\sum \def\sum{\displaystyle\!sum}
\let\!sup\sup \def\sup{\displaystyle\!sup}
\let\!inf\inf \def\inf{\displaystyle\!inf}
\let\!cap\cap \def\cap{\displaystyle\!cap}
\let\!max\max \def\max{\displaystyle\!max}
\let\!min\min \def\min{\displaystyle\!min}
\let\!frac\frac \def\frac{\displaystyle\!frac}
\catcode`!=12

\let\oldsection\section
\renewcommand\section{\setcounter{equation}{0}\oldsection}

\allowdisplaybreaks

\def\R{\mathbb{R}}

\def\N{\mathbb{N}}
\def\Z{\mathbb{Z}}
\def\Q{\mathbb{Q}}

\theoremstyle{plain}
\newtheorem{thm}{Theorem}[section]
\newtheorem{lem}[thm]{Lemma}
\newtheorem{conj}[thm]{Conjecture}
\newtheorem{cor}[thm]{Corollary}

\newtheorem{pro}[thm]{Proposition}
\theoremstyle{definition}

\newtheorem{re}[thm]{Remark}
\newtheorem{exa}[thm]{Example}

\begin{document}
\title{\bf Berndt-Type Integrals of Order Three and Series Associated with Jacobi Elliptic Functions}
\author{
{Hongyuan Rui${}^{a,}$\thanks{Email: rhy626514@163.com},\quad Ce Xu${}^{a,}$\thanks{Email: cexu2020@ahnu.edu.cn}\quad and\quad Jianqiang Zhao${}^{b,}$\thanks{Email: zhaoj@ihes.fr}}\\[1mm]
\small a. School of Mathematics and Statistics, Anhui Normal University,\\ \small  Wuhu 241002, P.R. China\\
\small b. Department of Mathematics, The Bishop's School, La Jolla, CA 92037, USA
}

\date{}
\maketitle

\noindent{\bf Abstract.} In this paper, we first establish explicit evaluations of six classes of hyperbolic sums by special values of the Gamma function by using the tools of the Fourier series expansions and the Maclaurin series expansions of a few Jacobi elliptic functions developed in our previous paper. Then, using the method of contour integrations involving hyperbolic and trigonometric functions, we establish explicit evaluations of two families of Berndt-type integrals of order three by special values of the Gamma function. Furthermore, we present some interesting consequences and illustrative examples.



\section{Introduction}

For nonnegative integer $a$ and positive integer $b$, the \emph{Berndt-type integrals of order $b$} are of the form
\begin{align}\label{BTI-definition-1}
\int_0^\infty \frac{x^{a}dx}{(\cos x\pm\cosh x)^b},
\end{align}
where $a\geq 0$ and $b\geq 1$ if the denominator has ``$+$' sign, and $a\geq 2b$ otherwise. This kind of integrals can be traced back to Ramanujan \cite[pp. 325-326]{Rama1916} who first submitted the integral
\begin{align}\label{inte-Ramanujan}
\int_0^\infty \frac{\sin(nx)}{x(\cos x+\cosh x)}dx=\frac{\pi}{4} \quad \text{(for any odd integer $n$)}
\end{align}
as a problem to the \emph{Indian Journal of Pure and Applied Mathematics}. Wilkinson \cite{W1916} provided a proof four years later. At the International Conference on Orthogonal Polynomials and $q$-Series, which was held in Orlando in May 2015 in celebration of the 70th birthday of Professor Mourad Ismail, Dennis Stanton gave a plenary talk titled ``A small slice of Mourad's work". One of the topics in that talk was about ``the mystery integral of Mourad Ismail" \cite{K2017}:
\begin{align}\label{inte-Mourad-Ismail}
\int_{-\infty}^\infty \frac{dx}{\cos(K \sqrt{x})+\cosh(K' \sqrt{x})}=2,
\end{align}
where $K(x)$ denotes the \emph{complete elliptic integral of the first kind} defined by
\begin{align}
K:=K(x):=K(k^2):=\int\limits_{0}^{\pi/2}\frac {d\varphi}{\sqrt{1-k^2\sin^2\varphi}}=\frac {\pi}{2} {_2}F_{1}\left(\frac {1}{2},\frac {1}{2};1;k^2\right).
\end{align}
Here $x=k^2$ and $k\ (0<k<1)$ is the modulus of $K$. The complementary modulus $k'$ is defined by $k'=\sqrt {1-k^2}$.
Furthermore, we set as usual
\begin{align}
K':=K(k'^2)=\int\limits_{0}^{\pi/2}\frac {d\varphi}{\sqrt{1-k'^2\sin^2\varphi}}=\frac {\pi}{2} {_2}F_{1}\left(\frac {1}{2},\frac {1}{2};1;1-k^2\right).
\end{align}
Similarly, the \emph{complete elliptic integral of the second kind} is denoted by (Whittaker and Watson \cite{WW1966})
\begin{align}
E:=E(x):=E(k^2):=\int\limits_{0}^{\pi/2} \sqrt{1-k^2\sin^2\varphi}d\varphi=\frac {\pi}{2} {_2}F_{1}\left(-\frac {1}{2},\frac {1}{2};1;k^2\right).\label{1.5}
\end{align}
This curious integral \eqref{inte-Mourad-Ismail} first appeared in \cite{Ismail1998} by Ismail and Valent. Berndt \cite{Berndt2016} provided direct evaluations of \eqref{inte-Ramanujan}, \eqref{inte-Mourad-Ismail} and other similar integrals of this type by using residue computations, the Fourier series expansions and the Maclaurin series expansions of Jacobi elliptic functions. In particular, he proved some explicit relations between hyperbolic series and integrals containing trigonometric and hyperbolic functions. For example, from \cite[Thm. 3.1 and Eq. (5.5)]{Berndt2016} for nonnegative integer $a$, Berndt showed that
\begin{align}\label{equ.one}
(1+(-1)^a)\int_0^\infty \frac{x^{2a+1}dx}{\cos x+\cosh x}=\frac{\pi^{2a+2}i^a}{2^a} \sum_{n=0}^\infty \frac{(-1)^n (2n+1)^{2a+1}}{\cosh((2n+1)\pi/2)},
\end{align}
and
\begin{align}\label{equ.-.one}
(1+i^{a+1})\int_0^\infty \frac{x^{a}dx}{\cos x-\cosh x}=2i(1+i)^{a-1}\pi^{a+1}\sum_{n=1}^\infty \frac{(-1)^{n+1}n^a}{\sinh(n\pi)}\quad (a\geq 2).
\end{align}
Then, applying the Fourier series expansions and the Maclaurin series expansions of a few Jacobi elliptic functions, Berndt proved that the hyperbolic series on the right-hand side of \eqref{equ.one} and \eqref{equ.-.one} can be expressed in terms of special values of the Gamma function in certain special cases. Some recent results on infinite series involving hyperbolic functions may also be found in the works of \cite{BB2002,Campbell,T2015,T2008,T2010,T2012,X2018,XuZhao-2022,Ya-2018}.

It turns out that Berndt's results and methods can be generalized and organized further. In our previous paper \cite{XZ2023}, by extending an argument as used in the proof of the main theorem of \cite{Berndt2016}, we proved that the following explicit evaluations of two Berndt-type integrals of order two (cf. \cite[Thm. 1.3]{XZ2023}):
\begin{align*}
&\int_0^\infty \frac{x^{4p+1}}{(\cos x-\cosh x)^2}dx\in \Q\frac{\Gamma^{8p}(1/4)}{\pi^{2p}}+\Q\frac{\Gamma^{8p+8}(1/4)}{\pi^{2p+4}},\\
&\int_0^\infty \frac{x^{4p+1}}{(\cos x+\cosh x)^2}dx\in \Q\frac{\Gamma^{8p}(1/4)}{\pi^{2p}}+\Q\frac{\Gamma^{8p+8}(1/4)}{\pi^{2p+4}},
\end{align*}
where $p\in \N$. Similarly, we also proved the following explicit evaluations of two Berndt-type integrals of order one
\begin{align*}
&\int_0^\infty \frac{x^{4p-1}}{\cos x-\cosh x}dx\in \Q \frac{\Gamma^{8p}(1/4)}{\pi^{2p}},\\
&\int_0^\infty \frac{x^{4p+1}}{\cos x+\cosh x}dx\in \Q \frac{\Gamma^{8p+4}(1/4)}{\pi^{2p+1}}.
\end{align*}

In this paper, we use the contour integrals and the Fourier series expansion and Maclaurin series expansion of Jacobi elliptic function to establish explicit evaluations of Berndt-type integrals of order three.
\begin{thm}\label{Main-one-theorem} Set $\Gamma:=\Gamma(1/4)$. For any positive integer $p$,
\begin{align*}
	\int_{0}^{\infty}\frac{z^{4p+1}\, dz}{(\cos z+\cosh z)^3}\in &\Q\frac{\Gamma^{8p-4}}{\pi^{2p-1}}+\Q\frac{\Gamma^{8p+4}}{\pi^{2p+3}}+\Q\frac{\Gamma^{8p+4}}{\pi^{2p+2}}+\Q\frac{\Gamma^{8p+4}}{\pi^{2p+1}}+\Q\frac{\Gamma^{8p+12}}{\pi^{2p+7}},
	\\\int_{0}^{\infty}\frac{z^{4p-1}\, dz}{\left(\cos z-\cosh z\right)^3}\in&\Q\frac{\Gamma^{8p-8}}{\pi^{2p-2}}+\Q\frac{\Gamma^{8p}}{\pi^{2p+2}}+\Q\frac{\Gamma^{8p}}{\pi^{2p+1}}+\Q\frac{\Gamma^{8p}}{\pi^{2p}}+\Q\frac{\Gamma^{8p+8}}{\pi^{2p+6}}\quad (p>1).
\end{align*}
\end{thm}

\section{Relations between Berndt-type Integrals and Series}

In this section, we will establish some explicit relations between the Berndt-type integrals \eqref{BTI-definition-1} with $b=3$ and some hyperbolic summations by using the contour integrations. To save space, throughout this section we will put
\begin{equation*}
\tn=\frac{2n-1}2.
\end{equation*}
We need the following lemma.
\begin{lem}\emph{(cf.\cite{XZ2023})} Let $n$ be an integer. Then we have
\begin{align}
&\frac{(-1)^n}
{{\cosh \left(\frac{1+i}{2}z \right)}}
\buildrel{z\to\tn\pi(1+i)}\over{=\joinrel=\joinrel=\joinrel=\joinrel=\joinrel=} 2i \left\{\frac1{1+i}\cdot\frac{1}{z-\tn\pi(1+i)}\atop +\sum_{k=1}^\infty (-1)^k \frac{{\bar \zeta}(2k)}{\pi^{2k}}\left(\frac{1+i}{2}\right)^{2k-1}\left(z-\tn\pi(1+i)\right)^{2k-1} \right\},\label{eq-cosh-1}\\
&\frac{(-1)^n}
{{\cosh \left(\frac{1-i}{2}z \right)}}
\buildrel{z \to \tn\pi (i-1)}\over{=\joinrel=\joinrel=\joinrel=\joinrel=\joinrel=}    2i \left\{\frac1{1-i}\cdot\frac{1}{z-\tn\pi(i-1)}\atop +\sum_{k=1}^\infty (-1)^k \frac{{\bar \zeta}(2k)}{\pi^{2k}}\left(\frac{1-i}{2}\right)^{2k-1}\left(z-\tn\pi(i-1)\right)^{2k-1} \right\},\label{eq-cosh-2}\\
&\frac{(-1)^n}
{{\sinh \left(\frac{1+i}{2}z \right)}}
\buildrel{z \to n\pi (1+i)}\over{=\joinrel=\joinrel=\joinrel=\joinrel=\joinrel=}   2\left\{\frac1{1+i}\cdot\frac{1}{z-n\pi(1+i)}\atop +\sum_{k=1}^\infty (-1)^k \frac{{\bar \zeta}(2k)}{\pi^{2k}}\left(\frac{1+i}{2}\right)^{2k-1}\left(z-n\pi(1+i)\right)^{2k-1} \right\},\label{eq-sinh-1}\\
&\frac{(-1)^n}
{{\sinh \left(\frac{i-1}{2}z \right)}}
\buildrel{z \to n\pi (1-i)}\over{=\joinrel=\joinrel=\joinrel=\joinrel=\joinrel=}   2\left\{\frac1{i-2}\cdot\frac{1}{z-n\pi(1-i)}\atop +\sum_{k=1}^\infty (-1)^k \frac{{\bar \zeta}(2k)}{\pi^{2k}}\left(\frac{i-1}{2}\right)^{2k-1}\left(z-n\pi(1-i)\right)^{2k-1} \right\}.\label{eq-sinh-2}
\end{align}
\end{lem}
Here  $\zeta\left(s\right)$ and $\bar{\zeta}\left(s\right)$ are \emph{Riemann zeta function} and \emph{alternating Riemann zeta function}, respectively, defined by
\begin{align*}
\zeta\left(s\right):=\sum_{n=1}^{\infty}\frac{1}{n^s}\quad \text{and}\quad \bar{\zeta}\left(s\right):=\sum_{n=1}^{\infty}\frac{(-1)^{n-1}}{n^s}\quad\left(\Re\left(s\right)>1\right).
\end{align*}
For even $s=2m\ (m\in \N)$ Euler proved the famous formula
\begin{align}
\zeta(2m) = -\frac12\frac{B_{2m}}{(2m)!}(2\pi i)^{2m},
\end{align}
where $B_{2m}$ are \emph{Bernoulli numbers} defined by the generating function
\begin{equation*}
 \frac{x}{e^x-1}=\sum_{n=0}^\infty \frac{B_n}{n!}x^n.
\end{equation*}

\begin{thm}\label{cos+cosh}
For any integer $a\geq 0$,
\begin{align*}
\frac{\left(1-i^{a+1}\right)}{\pi^{a-1}\left(1+i\right)^{a-1}}\int_{0}^{\infty}\frac{x^{a}\, dx}{\left(\cos x+\cosh x\right)^3}
&=-\frac{a\left(a-1\right)}{4}\sum_{n=1}^{\infty}\frac{(-1)^{n}\tilde{n}^{a-2}}{\cosh^3\left(\tilde{n}\pi\right)}
-\frac{5\pi^2}{2}\sum_{n=1}^{\infty}\frac{(-1)^{n}\tilde{n}^{a}}{\cosh^3\left(\tilde{n}\pi\right)}\nonumber\\
 + &\frac{3a\pi}{2}
\sum_{n=1}^{\infty}\frac{(-1)^{n}\tilde{n}^{a-1}\sinh\left(\tilde{n}\pi\right)}{\cosh^4\left(\tilde{n}\pi\right)}+3\pi^2
\sum_{n=1}^{\infty}\frac{(-1)^{n}\tilde{n}^{a}}{\cosh^5\left(\tilde{n}\pi\right)}.
\end{align*}
\end{thm}
\begin{proof}
Let $z=x+iy$ for $x,y\in \R$. Consider
\begin{align*}
\lim\limits_{R\to\infty}\int_{C_R}\frac{z^a\, dz}{(\cos z+\cosh z)^3}=\lim\limits_{R\to\infty}\int_{C_R}F\left(z\right)\, dz,
\end{align*}
where $C_R$ denotes the positively oriented quarter-circular contour consisting of the interval $[0,R]$, the quarter-circle $\Gamma_R$ with $|z|=R$ and $0\leq \arg z\leq \pi/2$, and $[iR,0]$ (i.e., the line segment from $i R$ to $0$ on the imaginary axis).

Clearly, there exist poles of order $3$ when
\begin{align*}
\cos z+\cosh z=2\cos\left\{\frac1{2}(z+iz)\right\}\cos\left\{\frac1{2}(z-iz)\right\}=0.
\end{align*}
Hence, there exist one set of poles at
\begin{align*}
z_n:=\frac{(2n-1)\pi i}{1+i}=\tn (1+i)\pi ,\quad n\in \Z,
\end{align*}
and
\begin{align*}
s_n:=\frac{(2n-1)\pi i}{1-i}=\tn(i-1)\pi ,\quad n\in \Z.
\end{align*}
The only poles lying inside $C_R$ are $z_n,\ n\geq 1,\ |z_n|<R$. From \eqref{eq-cosh-1}, we have
\begin{align*}
\frac{1}{\cosh^3\left(\frac{1+i}{2}z\right)}\buildrel{z\rightarrow z_n}\over{=\joinrel=}
(-1)^n i^3\left[\left(\frac{2}{1+i}\right)^3\frac{1}{\left(z-z_n\right)^3}-\frac{1}{1+i}\frac{1}{z-z_n}+o\left(1\right)\right].
\end{align*}
Hence, if $z\rightarrow z_n$ then
\begin{align*}
F\left(z\right)\buildrel{z\rightarrow z_n}\over{=\joinrel=}
\frac{z^a}{8\cosh^3\left(\frac{1-i}{2}z\right)}\left\{(-1)^n i^3\left[\left(\frac{2}{1+i}\right)^3\frac{1}{\left(z-z_n\right)^3}-\frac{1}{1+i}\frac{1}{z-z_n}+o\left(1\right)\right]\right\}.
\end{align*}
The residue $\underset{z=z_n}\Res \{F(z)\}$ at such a pole $z_n$ is give by
\begin{align*}
\underset{z=z_n}\Res \{F(z)\}=&\lim\limits_{z\to z_n}\frac{d^2}{\, dz^2}\frac{z^a(-1)^ni^3}{2^4\cosh^3\left(\frac{1-i}{2}z\right)}\left[\left(\frac{2}{1+i}\right)^3
-\frac{\left(z-z_n\right)^2}{1+i}\right]\\=&-\frac{(-1)^n\left(1-i\right)\left(1+i\right)^{a-2}\left(\tilde{n}\pi\right)^{a-2}\left(a-1\right)a}{2^3\cosh^3\left(\tilde{n}\pi\right)}
-\frac{(-1)^n5\left(1-i\right)\left(1+i\right)^{a-2}\left(\tilde{n}\pi\right)^{a}}{4\cosh^3\left(\tilde{n}\pi\right)}\\&+\frac{(-1)^n3\left(1-i\right)\left(1+i\right)^{a-2}
\left(\tilde{n}\pi\right)^{a-1}a\sinh\left(\tilde{n}\pi\right)}{4\cosh^4\left(\tilde{n}\pi\right)}+\frac{(-1)^n3\left(1-i\right)\left(1+i\right)^{a-2}\left(\tilde{n}\pi\right)^{a}}{2\cosh^5\left(\tilde{n}\pi\right)}.
\end{align*}
Note that $z=i y$ for the integral over $[iR,0]$. Hence,
\begin{align*}
\int_{iR}^{0}\frac{z^a\, dz}{(\cos z+\cosh z)^3}=-i^{a+1}\int_{0}^{R}\frac{y^a\, dy}{\left(\cos y+\cosh y\right)^3}.
\end{align*}
It is easy to show that, as $R\rightarrow \infty$
\begin{align*}
\lim\limits_{R\to\infty}\int_{\Gamma_R}\frac{z^a\, dz}{(\cos z+\cosh z)^3}=o\left(1\right).
\end{align*}
Applying the residue theorem, letting $R\rightarrow \infty$, we conclude that
\begin{align*}
\left(1-i^{a+1}\right)\int_{0}^{\infty}\frac{x^a\, dx}{\left(\cos x+\cosh x\right)^3}=2\pi i\sum_{n=1}^{\infty}\underset{z=z_n}\Res \{F(z)\}.
\end{align*}
This completes the proof of the theorem.
\end{proof}

\begin{thm} \label{cos-cosh} For any integer $a\ge6$, we have
\begin{multline*}
\frac{\left(1+i^{a+1}\right)}{\pi^{a-1}(1+i)^{a-3}}\int_{0}^{\infty}\frac{x^a\, dx}{(\cos x-\cosh x)^3}
=-\binom{a}{2}\sum_{n=1}^{\infty}\frac{(-1)^nn^{a-2}}{\sinh^3(n\pi)}
+3a\pi\sum_{n=1}^{\infty}\frac{(-1)^n\cosh(n\pi)n^{a-1}}{\sinh^4(n\pi)} \\
-5\pi^2\sum_{n=1}^{\infty}\frac{(-1)^nn^{a}}{\sinh^3(n\pi)}-6\pi^2\sum_{n=1}^{\infty}\frac{(-1)^nn^{a}}{\sinh^5(n\pi)}.
\end{multline*}
\end{thm}
\begin{proof}
Let $z=x+iy$ for $x,y\in \R$. Consider
\begin{align*}
\lim\limits_{R\to\infty}\int_{C_R}\frac{z^a\, dz}{\left(\cos z-\cosh z\right)^3}=\lim\limits_{R\to\infty}\int_{C_R}G\left(z\right)\, dz,
\end{align*}
where $C_R$ denotes the positively oriented quarter-circular contour consisting of the interval $[0,R]$; the quarter-circle $\Gamma_R$ with $|z|=R$ and $0\leq \arg\leq \pi/2$; and the segment $[i R,0]$ on the imaginary axis.
Clearly, there exist poles of order $3$ in the present case when
\begin{align*}
\cos z-\cosh z=-2\sin\left\{\frac1{2}(z+iz)\right\}\sin\left\{\frac1{2}(z-iz)\right\}=0.
\end{align*}
Hence, there exist one set of poles at $t_n=n\pi(1-i)$ and $y_n=n\pi(1+i)$, where $n\in\N$. Those lying on the interior of $C_R$ are $y_n=n\pi (1+i),n\geq 1,|y_n|<R$. Hence, by \eqref{eq-sinh-1}, we have
\begin{align*}
G(z)\buildrel{z\rightarrow y_n}\over{=\joinrel=}\frac{(-1)^nz^a}{8\sinh^3\left(\frac{i-1}{2}z\right)} \left\{\frac{8}{(1+i)^3}\frac{1}{(z-y_n)^3}-\frac{1}{z-y_n}\frac{1}{1+i}\right\}.
\end{align*}
Further, the residue is
\begin{align*}
\underset{z=y_n}\Res  \{G(z)\}=&\lim\limits_{z\to y_n}\frac{d^2}{\, dz^2}\frac{z^a(-1)^n}{2^4\sinh^3\left(\frac{i-1}{2}z\right)}\left[\left(\frac{2}{1+i}\right)^3
-\frac{\left(z-y_n\right)^2}{1+i}\right]\\=&-\frac{(-1)^na\left(a-1\right)(n\pi)^{a-2}\left(1+i\right)^{a-4}\left(1-i\right)}{4\sinh^3(n\pi)}
-\frac{(-1)^n3(n\pi)^{a}\left(1+i\right)^{a-4}\left(1-i\right)}{\sinh^5(n\pi)}\\&
-\frac{(-1)^n5(n\pi)^{a}\left(1+i\right)^{a-4}\left(1-i\right)}{2\sinh^3(n\pi)}
+\frac{(-1)^n3a(n\pi)^{a-1}\left(1+i\right)^{a-4}\left(1-i\right)\cosh(n\pi)}{2\sinh^4(n\pi)}.
\end{align*}
With $z=iy$, we see that the integral over $[iR,0]$
\begin{align*}
\int_{iR}^{0}\frac{z^a\, dz}{\left(\cos z-\cosh z\right)^3}=i^{a+1}\int_{0}^{R}\frac{y^a\, dy}{(\cos y-\cosh y)^3}.
\end{align*}
Applying the residue theorem, as $R\rightarrow\infty$, we conclude that
\begin{align*}
\left(1+i^{a+1}\right)\int_{0}^{\infty}\frac{x^a\, dx}{(\cos x-\cosh x)^3}
=2\pi i\sum_{n=1}^{\infty}\underset{z=y_n}\Res  \{G(z)\}.
\end{align*}
This completes the proof of the theorem.
\end{proof}

Setting $a=4p + 1$ and $a = 4p - 1$ in Theorems \ref{cos+cosh} and \ref{cos-cosh}, we get the following corollaries.
\begin{cor}\label{corcos+cos}
For any integer $p\ge 0$, then
\begin{align*}
\frac{(-4)^{p+1}}{\pi^{4p}}\int_{0}^{\infty}\frac{x^{4p+1}\, dx}{(\cos x+\cosh x)^3}
=& 4p(4p+1)\sum_{n=1}^{\infty}\frac{(-1)^{n}(2n-1)^{4p-1}}{\cosh^3(\tilde{n}\pi)}
 +\frac{5\pi^2}{2}\sum_{n=1}^{\infty}\frac{(-1)^{n}(2n-1)^{4p+1}}{\cosh^3(\tilde{n}\pi)}\\
-3(4p+1)\pi &\sum_{n=1}^{\infty}\frac{(-1)^{n}(2n-1)^{4p}\sinh(\tilde{n}\pi)}{\cosh^4(\tilde{n}\pi)}
 -3\pi^2\sum_{n=1}^{\infty}\frac{(-1)^{n}(2n-1)^{4p+1}}{\cosh^5(\tilde{n}\pi)}.
\end{align*}
\end{cor}

\begin{cor}\label{corcos-cosh}For any integer $p\ge2$, then
\begin{align*}
&\frac{(-1)^{p-1}}{\pi^{4p-2}2^{2p-3}}\int_{0}^{\infty}\frac{x^{4p-1}\, dx}{(\cos x-\cosh x)^3}\\
=&-\frac{(4p-1)(4p-2)}{2}\sum_{n=1}^{\infty}\frac{(-1)^nn^{4p-3}}{\sinh^3(n\pi)}
 +3(4p-1)\pi\sum_{n=1}^{\infty}\frac{(-1)^n\cosh(n\pi)n^{4p-2}}{\sinh^4(n\pi)}\\
&-5\pi^2\sum_{n=1}^{\infty}\frac{(-1)^nn^{4p-1}}{\sinh^3(n\pi)}-6\pi^2\sum_{n=1}^{\infty}\frac{(-1)^nn^{4p-1}}{\sinh^5(n\pi)}.
\end{align*}
\end{cor}

Therefore, in order to obtain the evaluations of the integrals on the left hand sides of Corollary \ref{corcos+cos} and \ref{corcos-cosh}, we will need to establish some explicit evaluations of the six classes of hyperbolic sums:
\begin{align*}
&\sum_{n=1}^{\infty}\frac{(-1)^{n}(2n-1)^{2p-1}}{\cosh^3(\tilde{n}\pi)},\quad \sum_{n=1}^{\infty}\frac{(-1)^{n}(2n-1)^{4p}\sinh(\tilde{n}\pi)}{\cosh^4(\tilde{n}\pi)},\quad \sum_{n=1}^{\infty}\frac{(-1)^{n}(2n-1)^{4p+1}}{\cosh^5(\tilde{n}\pi)}
\end{align*}
and
\begin{align*}
\sum_{n=1}^{\infty}\frac{(-1)^nn^{2p-1}}{\sinh^3(n\pi)},\quad \sum_{n=1}^{\infty}\frac{(-1)^n\cosh(n\pi)n^{4p-2}}{\sinh^4(n\pi)},\quad \sum_{n=1}^{\infty}\frac{(-1)^nn^{4p-1}}{\sinh^5(n\pi)}.
\end{align*}
In next section, we will give the explicit evaluations of above six classes of hyperbolic sums by applying the Fourier series expansions and the Maclaurin series expansions of a few Jacobi elliptic functions.

\section{Evaluations of hyperbolic summations via Jacobi functions}

In this section, we will apply the Fourier series expansions and power series expansions of Jacobi elliptic functions ${\rm sd}$ and ${\rm sn}$ to establish some explicit evaluations of hyperbolic summations.

Recall that the Jacobi elliptic function $\sn (u)=\sn (u,k)$ is defined via the inversion of the elliptic integral
\begin{align}
u=\int_0^\varphi \frac{dt}{\sqrt{1-k^2\sin^2 t}}\quad (0<k^2<1),
\end{align}
namely, $\sn (u):=\sin \varphi$. As before, we refer $k={\rm mod}\, u$ as the elliptic modulus. We also write $\varphi={\rm am}(u,k)={\rm am}(u)$ and call it the Jacobi amplitude. Then the Jacobi elliptic functions $\cn u$ and $\dn u$ may be defined by
\begin{align*}
&\cn (u):=\sqrt{1-\sn^2 (u)}\quad\text{and} \quad  \dn (u):=\sqrt{1-k^2\sn^2(u)}.
\end{align*}
Further, we define the Jacobi elliptic function ${\rm sd}$ by
\begin{align*}
{\rm sd}:=\frac{\sn (u)}{\dn (u)}.
\end{align*}

In order to better state our main results, we shall henceforth adopt the notations of Ramanujan (see Berndt's book \cite{B1991}). Let
\begin{align}\label{rel-den}
x:=k^2,\quad y:=y(x):=\pi \frac {K'}{K}, \quad q:=q(x):=e^{-y},\quad z:=z(x):=\frac {2}{\pi}K,\quad z':=\frac{\, dz}{dx}.
\end{align}
Using the identity $(a)_{n+1}=a(a+1)_n$, it is easily shown that
\[\frac {d}{dx}{_2}F_1(a,b;c;x)=\frac{ab}{c}{_2}F_1(a+1,b+1;c+1;x)\]
and more generally,
\begin{align}
\frac {d^n}{dx^n}{_2}F_1(a,b;c;x)=\frac{(a)_n(b)_n}{(c)_n}{_2}F_1(a+n,b+n;c+n;x)\quad(n\in\N_0).\label{1.7}
\end{align}
Then,
\begin{align}
\frac {d^nz}{dx^n}=\frac{(1/2)^2_n}{n!}{_2}F_1\left(\frac{1}{2}+n,\frac{1}{2}+n;1+n;x\right).
\end{align}
Applying the identity (see  \cite[Theorem 3.5.4(i)]{A2000})
\begin{align}
{_2}F_1\left(a,b;\frac {a+b+1}{2};\frac 1{2}\right)=\frac{\Gamma\left(\frac{1}{2}\right)\Gamma\left(\frac{a+b+1}{2}\right)}{\Gamma\left(\frac{a+1}{2}\right)\Gamma\left(\frac{b+1}{2}\right)},
\end{align}
we can show by an elementary calculation that
\begin{align}\label{den-z-diff}
\frac {d^nz}{dx^n}\bigg|_{x=1/2}=\frac{(1/2)^2_n\sqrt{\pi}}{\Gamma^2\left(\frac{n}{2}+\frac {3}{4}\right)}.
\end{align}
In particular, taking $x=1/2,n=0,1,2,3,4$ in \eqref{den-z-diff}, we obtain that
\begin{align}\label{special-case-x}
y=\pi,\quad & z\Big(\frac 1{2}\Big)=\frac{\Gamma^2(1/4)}{2\pi^{3/2}},\quad z'\Big(\frac 1{2}\Big)=\frac {4\sqrt{\pi}}{\Gamma^2(1/4)},\quad z''\Big(\frac 1{2}\Big)=\frac{\Gamma^2(1/4)}{2\pi^{3/2}},\nonumber\\
&z^{\left(3\right)}\left(\frac{1}{2}\right)=\frac{36\sqrt{\pi}}{\Gamma^2(1/4)},\quad z^{\left(4\right)}\left(\frac{1}{2}\right)=\frac{25\Gamma^2(1/4)}{2\pi^{3/2}},
\end{align}
where we have used the two classical relations
\[\Gamma(x+1)=x\Gamma(x)\quad{\rm and}\quad \Gamma(x)\Gamma(1-x)=\frac{\pi}{\sin(\pi x)}.\]

To begin with, we record the Fourier series expansions of Jacobi elliptic functions sd$\left(u\right)$ and sn$\left(u\right)$ (see \cite{DCLMRT1992})
\begin{align}\label{sd}
&\text{sd}\left(u\right)=\frac{2\pi}{Kkk'}\sum_{n=0}^{\infty}\frac{(-1)^nq^{n+1/2}}{1+q^{2n+1}}\sin\left[(2n+1)\frac{\pi u}{2K}\right],\\
&\text{sn}\left(u\right)=\frac{2\pi}{Kk}\sum_{n=0}^{\infty}\frac{q^{n+1/2}}{1-q^{2n+1}}\sin\left[(2n+1)\frac{\pi u}{2K}\right].
\end{align}

\subsection{Part I}

\begin{lem}
The Maclaurin series expansion of ${\rm sd}(u)$ has the form
\begin{equation}\label{equ:cor-sd}
    \sum_{m\ge0}\frac{p_{2m+1}(x)}{\left(2m+1\right)!}u^{2m+1}, \quad\text{where}\quad p_{2m+1}(x)\in\Z[x] \ \forall m\ge 0.
\end{equation}
\end{lem}
\begin{proof}
The proof is completely similar to the proof of \cite[Lemma 7.1]{XZ2023} and is thus omitted.
\end{proof}

\begin{thm} For any integer $m\ge0$,
\begin{align}\label{equ-sd-coffoc}		
\sum_{n=0}^{\infty}\frac{(-1)^n(2n+1)^{2m+1}}{\cosh\left( (2n+1)y/2\right)}=\frac{(-1)^m}{2}z^{2m+2}\sqrt{x(1-x)}p_{2m+1}(x)\in\sqrt{x(1-x)}\mathbb{Q}\left[x,z\right].
\end{align}
\end{thm}
\begin{proof}
Applying \eqref{rel-den}, the (\ref{sd}) can be rewritten as
\begin{align}\label{sd2}		 \text{sd}\left(u\right)=&\frac{4}{z\sqrt{x(1-x)}}\sum_{n=0}^{\infty}\frac{(-1)^ne^{-(2n+1)y/2}}{1+e^{-(2n+1)y}}\sin\left[(2n+1)\frac{u}{z}\right]\nonumber\\=&\sum_{m=0}^{\infty}\left\{(-1)^m4\sum_{n=0}^{\infty}\frac{(-1)^n(2n+1)^{2m+1}}{e^{(2n+1)y/2}+e^{-(2n+1)y/2}}\right\}\frac{u^{2m+1}}{\sqrt{x(1-x)}z^{2m+2}\left(2m+1\right)!}.
\end{align}
Comparing the coefficients of $u^{2m+1}$ in \eqref{sd2} and \eqref{equ:cor-sd}, we can derive \eqref{equ-sd-coffoc}	easily.
\end{proof}

\begin{exa} By \emph{Mathematica}, comparing the coefficients of $u^1,u^3$ and $u^5$, we obtain
\begin{align*}
\sum_{n=0}^{\infty}\frac{(-1)^n(2n+1)}{\cosh\left((2n+1)y/2\right)}=&\frac{1}{2}z^{2}\sqrt{x(1-x)},\\
\sum_{n=0}^{\infty}\frac{(-1)^n(2n+1)^3}{\cosh\left( (2n+1)y/2 \right)}=&-\frac{1}{2}z^{4}(2x-1)\sqrt{x(1-x)},\\
\sum_{n=0}^{\infty}\frac{(-1)^n(2n+1)^5}{\cosh\left( (2n+1)y/2 \right)}=&\frac{1}{2}z^{6}\left(16x^2-16x+1\right)\sqrt{x(1-x)}.
	\end{align*}
\end{exa}

Noting the fact that
\begin{align*}	
\frac{d^2}{dx^2}\sum_{n=0}^{\infty}\frac{(-1)^n(2n+1)^{2m+1}}{\cosh\left( (2n+1)y/2\right)}
=&\frac{(y')^2}{4}
\sum_{n=0}^{\infty}\frac{(-1)^n(2n+1)^{2m+3}}{\cosh\left((2n+1)y/2\right)}-\frac{(y')^2}{2}\sum_{n=0}^{\infty}\frac{(-1)^n(2n+1)^{2m+3}}
{\cosh^3\left((2n+1)y/2\right)}\\
&-\frac{y''}{2}\sum_{n=0}^{\infty}\frac{(-1)^n(2n+1)^{2m+2}\sinh\left((2n+1)y/2\right)}{\cosh^2\left((2n+1)y/2\right)},
\end{align*}
we obtain the following corollary.

\begin{cor}\label{cor-coshWithq} For any integer $m\ge 0$,
\begin{align*}	 \sum_{n=0}^{\infty}\frac{(-1)^n(2n+1)^{2m+3}}{\cosh^3\left((2n+1)y/2\right)}=&-\frac{(-1)^m}{4}z^{2m+4}\sqrt{x(1-x)}p_{2m+3}(x)\\&+\frac{y''}{\left(y'\right)^3}\frac{d}{dx}\left\{(-1)^mz^{2m+2}\sqrt{x(1-x)}p_{2m+1}(x)\right\}\\&-\frac{4}{\left(y'\right)^2}\frac{d^2}{dx^2}\left\{\frac{(-1)^m}{4}z^{2m+2}\sqrt{x(1-x)}p_{2m+1}(x)\right\}.
\end{align*}
\end{cor}

\begin{re}\label{re-hy-cosh} Obviously, for positive integer $p$ we have
\begin{align*}		
&\sum_{n=1}^{\infty}\frac{(-1)^{n}\left(2n-1\right)^{2p+1}}{\cosh^3\left(\tilde{n}y\right)}
=-2\left(\frac{dx}{dy}\right)^2\frac{d^2}{dx^2}\sum_{n=1}^{\infty}\frac{(-1)^{n}\left(2n-1\right)^{2p-1}}{\cosh\left(\tilde{n}y\right)}
+\frac{1}{2}\sum_{n=1}^{\infty}\frac{(-1)^{n}\left(2n-1\right)^{2p+1}}{\cosh\left(\tilde{n}y\right)}\\
&\phantom{\sum_{n=1}^{\infty}\frac{(-1)^{n}\left(2n-1\right)^{2p+1}}{\cosh^3\left(\tilde{n}y\right)}=}
-\left(\frac{dx}{dy}\right)^2\frac{d^2y}{dx^2}\sum_{n=1}^{\infty}\frac{(-1)^{n}\left(2n-1\right)^{2p}\sinh\left(\tilde{n}y\right)}{\cosh^2\left(\tilde{n}y\right)},\\		 &\sum_{n=1}^{\infty}\frac{(-1)^{n}\left(2n-1\right)^{2p+2}\sinh\left(\tilde{n}y\right)}{\cosh^4\left(\tilde{n}y\right)}
=-\frac{2}{3}\frac{dx}{dy}\frac{d}{dx}\sum_{n=1}^{\infty}\frac{(-1)^{n}\left(2n-1\right)^{2p+1}}{\cosh^3\left(\tilde{n}y\right)},
\end{align*}
and
\begin{align*}		
\sum_{n=1}^{\infty}\frac{(-1)^{n}\left(2n-1\right)^{2p+3}}{\cosh^5\left(\tilde{n}y\right)}
=&-\frac{1}{3}\left(\frac{dx}{dy}\right)^2\frac{d^2}{dx^2}\sum_{n=1}^{\infty}\frac{(-1)^{n}\left(2n-1\right)^{2p+1}}{\cosh^3\left(\tilde{n}y\right)}+\frac{3}{4}
\sum_{n=1}^{\infty}\frac{(-1)^{n}\left(2n-1\right)^{2p+3}}{\cosh^3\left(\tilde{n}y\right)}\\
&-\frac{1}{2}\left(\frac{dx}{dy}\right)^2\frac{d^2y}{dx^2}
\sum_{n=1}^{\infty}\frac{(-1)^{n}\left(2n-1\right)^{2p+2}\sinh\left(\tilde{n}y\right)}{\cosh^4\left(\tilde{n}y\right)},
\end{align*}
where $dx/dy=-x(1-x)z^2$ (see \cite[P. 120, Entry. 9(i)]{B1991}).
\end{re}

\begin{exa} By \emph{Mathematica}, we have
\begin{align*}
&\sum_{n=1}^{\infty}\frac{(-1)^{n-1}\left(2n-1\right)^{3}}{\cosh^3\left(\tilde{n}y\right)}\\
&=-\frac{1}{4}\sqrt{x(1-x)}z^4\begin{Bmatrix}
		 2x-1+z^2\left(8x^2-8x+1\right)+24(x-1)^2x^2(z')^2\\+4(x-1)xz\left[5(2x-1)z'-2(x-1)xz''\right]
	\end{Bmatrix},\\
&\sum_{n=1}^{\infty}\frac{(-1)^{n-1}\left(2n-1\right)^{4}\sinh\left(\tilde{n}y\right)}{\cosh^4\left(\tilde{n}y\right)}\\
&=\frac{1}{12}\sqrt{x(1-x)}z^5\begin{Bmatrix}	 \left(48x^3-72x^2+26x-1\right)z^3\\+8(x-1)xz'\left(24(z')^2x^2(x-1)^2+2x-1\right)\\+z\left(8x^2-8x+1+320(x-1)^2x^2(2x-1)(z')^2\right.\\\left.+176(x-1)^3x^3z'z''\right)+8(x-1)xz^2\\\times\left(\left(52x^2-52x+9\right)z'+2(x-1)x\right.\\\left.\times\left(5(2x-1)z''+(x-1)xz^{\left(3\right)}\right)\right)
	\end{Bmatrix},\\
&\sum_{n=1}^{\infty}\frac{(-1)^{n-1}\left(2n-1\right)^{5}}{\cosh^5\left(\tilde{n}y\right)}\\
&=\frac{1}{48}\sqrt{x(1-x)}z^6\begin{Bmatrix}
		 9\left(16x^2-16x+1\right)\\+\left(384x^4-768x^3+464x^2-80x+1\right)z^4\\+800(x-1)^2x^2(2x-1)(z')^2+1920(x-1)^4x^4(z')^4\\+8(x-1)xz\left[5\left(56x^2-56x+9\right)z'\right.\\+648(x-1)^2x^2(2x-1)(z')^3+20x\left(2x^2-3x+1\right)z''\\\left.+408(x-1)^3x^3(z')^2z''\right]\\+2z^2\left[5\left(48x^3-72x^2+26x-1\right)\right.\\+8(x-1)^2x^2\left(844x^2-844x+163\right)(z')^2\\+176(x-1)^4x^4\left(z''\right)^2+8(x-1)^3x^3z'\\\left.\times\left(227(2x-1)z''+36(x-1)xz^{\left(3\right)}\right)\right]\\+8(x-1)xz^3\left[\left(616x^3-924x^2+366x-29\right)z'\right.\\+2(x-1)x\left(2\left(86x^2-86x+17\right)z''\right.\\\left.\left.+(x-1)x\left(17(2x-1)z^{\left(3\right)}+2(x-1)xz^{\left(4\right)}\right)\right)\right]
	\end{Bmatrix}.
\end{align*}
\end{exa}

\begin{thm}\label{cosh} Set $\Gamma=\Gamma(1/4)$ and $p_{n}^{(k)}=p_{n}^{(k)}(1/2)$ for all $n$ and $k$. Then,
for any integer $m>0$ we have
\begin{align}
&\begin{aligned}
\sum_{n=0}^{\infty}\frac{(-1)^n(2n+1)^{4m+1}}{\cosh^3((2n+1)\pi/2)}=\frac{\Gamma^{8m+4}}{2^{4m+5}\pi^{6m+4}}\left[\pi p_{4m+1}+(4m+1)p'_{4m-1}\right],	
\end{aligned}\\
&\begin{aligned}
\sum_{n=0}^{\infty}\frac{(-1)^n(2n+1)^{4m-1}}{\cosh^3((2n+1)\pi/2)}
=-\frac{\Gamma^{8m-4}}{2^{4m+7}\pi^{6m+3}}\begin{Bmatrix}128\left(8m^2-6m+1\right)\pi^4 p_{4m-3}\\+\Gamma^8\left[(4m-6)p_{4m-3}+p''_{4m-3}\right]\end{Bmatrix},
\end{aligned}\\
&\begin{aligned}
&\sum_{n=0}^{\infty}\frac{(-1)^n(2n+1)^{4m}\sinh((2n+1)\pi/2)}{\cosh^4((2n+1)\pi/2)}\\&=\frac{\Gamma^{8m-4}}{3\cdot2^{4m+6}\pi^{6m}}\begin{Bmatrix}
-256m\left(8m^2-6m+1\right) p_{4m-3}\\
-\frac{\Gamma^8}{\pi^4}\left[\pi p'_{4m-1}+6m\right.\left.\left((4m-6)p_{4m-3}+p''_{4m-3}\right)\right]\end{Bmatrix}, \end{aligned}\\
&\begin{aligned}
&\sum_{n=0}^{\infty}\frac{(-1)^n(2n+1)^{4m+1}}{\cosh^5((2n+1)\pi/2)}\\&=\frac{\Gamma^{8m-4}}{3\cdot2^{4m+15}\pi^{6m+9}}\begin{Bmatrix}
32768m(2m-1)(4m-1)(4m+1)\pi^8p_{4m-3}\\+256\pi^4\Gamma^8\Big[9\pi^2 p_{4m+1}+10(4m+1)\pi p'_{4m-1}\\
 +6m(4m+1)\left((4m-6)p_{4m-3}+p''_{4m-3}\right)\Big]\\
+\Gamma^{16}\Big[8(m-2)(6m-7) p_{4m-3}
+12(2m-5)p''_{4m-3}+p^{\left(4\right)}_{4m-3}\Big]\end{Bmatrix}.
\end{aligned}
\end{align}
\end{thm}
\begin{proof}
We have the following classical result (see the bottom of page 47 of Hancock's book \cite{Hancock1910}):
\begin{equation*}
\sn(iu,\sqrt{1-k^2})=i\frac{\sn(u,k)}{\cn(u,k)},\quad   \dn(iu,\sqrt{1-k^2})=\frac{\dn(u,k)}{\cn(u,k)},
\end{equation*}
where $i=\sqrt{-1}$ as usual. Therefore
\begin{align*}
i\text{sd}\left(iu,\sqrt{1-k^2}\right)+\text{sd}\left(u,k\right)=0.	
\end{align*}
Using \eqref{equ:cor-sd}, this implies that for all $m \ge 0$
\begin{align*}
&p'_{4m+1}(1-x)+p'_{4m+1}(x)=0,\quad p^{\left(3\right)}_{4m+1}(1-x)+p^{\left(3\right)}_{4m+1}(x)=0.
\end{align*}
And for $m>0$,
\begin{align*}
&p'_{4m-3}(1-x)+p'_{4m-3}(x)=0,\quad p^{\left(3\right)}_{4m-3}(1-x)+p^{\left(3\right)}_{4m-3}(x)=0,\quad p_{4m-1}(1-x)+p_{4m-1}(x)=0,\\
&p''_{4m-1}(1-x)+p''_{4m-1}(x)=0,\quad p^{\left(4\right)}_{4m-1}(1-x)+p^{\left(4\right)}_{4m-1}(x)=0.
\end{align*}
Taking $x = 1/2$ yields
\begin{align*}
	 p_{4m-1}(1/2)=0=p'_{4m+1}(1/2)=p'_{4m-3}(1/2)=p''_{4m-1}(1/2)=p^{\left(4\right)}_{4m-1}(1/2)=0.
\end{align*}
Hence, the theorem follows from Corollary \ref{cor-coshWithq}, Remark \ref{re-hy-cosh} and \eqref{special-case-x}.
\end{proof}

\begin{exa} Set $\Gamma=\Gamma(1/4)$. By \emph{Mathematica}, we have
	\begin{align*}
		 \sum_{n=0}^{\infty}&\frac{(-1)^n(2n+1)^{3}}{\cosh^3((2n+1)\pi/2)}=-\frac{3\Gamma^4}{2^4\pi^5}+\frac{\Gamma^{12}}{2^{10}\pi^9},\\
		 \sum_{n=0}^{\infty}&\frac{(-1)^n(2n+1)^{5}}{\cosh^3((2n+1)\pi/2)}=-\frac{3\Gamma^{12}}{2^9\pi^{9}}+\frac{5\Gamma^{12}}{2^8\pi^{10}},\\
	 \sum_{n=0}^{\infty}&\frac{(-1)^n(2n+1)^{7}}{\cosh^3((2n+1)\pi/2)}=\frac{63\Gamma^{12}}{2^8\pi^{11}}-\frac{13\Gamma^{20}}{2^{14}\pi^{15}},\\
		 \sum_{n=0}^{\infty}&\frac{(-1)^n(2n+1)^{9}}{\cosh^3((2n+1)\pi/2)}=\frac{189\Gamma^{20}}{2^{13}\pi^{15}}-\frac{297\Gamma^{20}}{2^{12}\pi^{16}},
		 \\\sum_{n=0}^{\infty}&\frac{(-1)^n(2n+1)^{4}\sinh((2n+1)\pi/2)}{\cosh^4((2n+1)\pi/2)}=-\frac{\Gamma^{4}}{2^2\pi^{6}}-\frac{\Gamma^{12}}
{3\cdot2^9\pi^{9}}+\frac{\Gamma^{12}}{2^8\pi^{10}},
		 \\\sum_{n=0}^{\infty}&\frac{(-1)^n(2n+1)^{8}\sinh((2n+1)\pi/2)}{\cosh^4((2n+1)\pi/2)}
=\frac{21\Gamma^{12}}{2^5\pi^{12}}-\frac{13\Gamma^{20}}{2^{11}\pi^{16}}+\frac{11\Gamma^{20}}{2^{13}\pi^{15}},\\
		\sum_{n=0}^{\infty}&\frac{(-1)^n(2n+1)^{5}}{\cosh^5((2n+1)\pi/2)}
=\frac{5\Gamma^{4}}{2^4\pi^{7}}-\frac{5\Gamma^{12}}{2^9\pi^{11}}+\frac{25\Gamma^{12}}{3\cdot2^9\pi^{10}}
-\frac{9\Gamma^{12}}{2^{11}\pi^{9}}+\frac{\Gamma^{20}}{3\cdot2^{16}\pi^{15}},\\
		\sum_{n=0}^{\infty}&\frac{(-1)^n(2n+1)^{9}}{\cosh^5((2n+1)\pi/2)}
=-\frac{189\Gamma^{12}}{2^7\pi^{13}}+\frac{117\Gamma^{20}}{2^{12}\pi^{17}}-\frac{495\Gamma^{20}}{2^{13}\pi^{16}}
+\frac{567\Gamma^{20}}{2^{15}\pi^{15}}-\frac{\Gamma^{28}}{2^{16}\pi^{21}}.
	\end{align*}
\end{exa}

\subsection{Part II}
\begin{lem}\label{lem-sn-prow}\emph{(cf. \cite[Lemma 7.3]{XZ2023})}
There are polynomials $g_m(x)\in \Q[x]$ such that the Maclaurin series of $\sn (u,k)$ has the form
\begin{align*}
\sum_{m=1}^\infty \frac{g_m(x)}{m!}u^m\quad (x=k^2),
\end{align*}
where $g_{2k}(x)=0$ since $\sn (u,k)$ is an odd function.
\end{lem}
\begin{proof}
The lemma follows immediately by \cite[Lemma 7.3.]{XZ2023}.
\end{proof}

\begin{lem}\label{lem-coff-sn-pro}
Notation as before. There are polynomials where $q_{2m}(x)\in \Q[x]$ such that the Maclaurin series of $\sn^2(u,k)$ has the form
\begin{align*}
\sum_{n=1}^\infty \frac{q_{2n}(x)}{(2n)!}u^{2n},
\end{align*}
where
\begin{align}\label{equ-coffec-q}
q_{2n}(x):=\sum_{j=1}^n \binom{2n}{2j-1} g_{2j-1}(x)g_{2n-2j+1}(x).
\end{align}
\end{lem}
\begin{proof} By \emph{Cauchy product formula}
	\begin{align*}
		\sn^2(u,k)&=\left(\sum_{n=1}^\infty \frac{g_{2n-1}(x)}{(2n-1)!}u^{2n-1}\right)^2\\
		&=\sum_{n=1}^\infty \left(\sum_{j=1}^n \frac{g_{2j-1}(x)g_{2n-2j+1}(x)}{(2j-1)!(2n-2j+1)!}\right)u^{2n}=\sum_{n=1}^\infty \frac{q_{2n}(x)}{(2n)!}u^{2n}.
	\end{align*}
	This completes the proof of the lemma.
\end{proof}

\begin{lem}\label{lem-sn-pro} (\cite[P.881, {\bf 8.153}-4]{GR}) We have
\[\sn (u,k)=k^{-1} \sn (ku,k^{-1}).\]
\end{lem}

\begin{pro} Notation as before. For any positive integer $n$, we have
	\begin{align}
		&g_{2n-1}(x)=x^{n-1}g_{2n-1}(1/x),\label{equ-one}\\
		&q_{2n}(x)=x^{n-1}q_{2n}(1/x).\label{equ-two}
	\end{align}
\end{pro}
\begin{proof}
	To prove \eqref{equ-one}, by Lemmas \ref{lem-sn-prow} and \ref{lem-sn-pro}, one obtains
	\begin{align*}
		\sum_{n=1}^\infty \frac{g_{2n-1}(x)}{(2n-1)!}u^{2n-1}=\sum_{n=1}^\infty \frac{x^{n-1}g_{2n-1}(1/x)}{(2n-1)!}u^{2n-1}.
	\end{align*}
	Then, comparing the coefficients of $u^{2n-1}$ yields \eqref{equ-one}. To prove \eqref{equ-two}, from \eqref{equ-coffec-q} and \eqref{equ-one} we have
	\begin{align*}
		q_{2n}(x)&=\sum_{j=1}^n \binom{2n}{2j-1} x^{j-1}g_{2j-1}(1/x)x^{n-j}g_{2n-2j+1}(1/x)\\
		&=x^{n-1}\sum_{j=1}^n \binom{2n}{2j-1} g_{2j-1}(1/x)g_{2n-2j+1}(1/x)\\
		&=x^{n-1}q_{2n}(1/x).
	\end{align*}
	Thus, we have finished the proof of the proposition.
\end{proof}

\begin{pro}\label{pro-coeffic-qs}
Notation as before. For any positive integer $m$, we have
	\begin{align}
		&q_{4m}(-1)=0,\label{equ-coeff-q-one}\\
		&q'_{4m-2}(-1)+(m-1)q_{4m-2}(-1)=0,\label{equ-coeff-q-diff-one}\\
		&q''_{4m}(-1)+2(m-1)q'_{4m}(-1)=0\label{equ-coeff-q-diff-two}.
	\end{align}
\end{pro}
\begin{proof}
	Setting $n=2m$ and $x=-1$ in \eqref{equ-two} gives \eqref{equ-coeff-q-one}. Differentiating \eqref{equ-two} with respect to $x$ we have
	\begin{align}\label{equ-coeff-q-diff-relation}
		q'_{2n}(x)-(n-1)x^{n-2}q_{2n}(1/x)+x^{n-3}q'_{2n}(1/x)=0.
	\end{align}
	Then, letting $x=-1$ and $n=2m-1$ yields \eqref{equ-coeff-q-diff-one}. Further, differentiating \eqref{equ-coeff-q-diff-relation} with respect to $x$, we obtain
	\begin{align}\label{equ-coeff-q-diff-relation-two}
		q''_{2n}(x)-(n-1)(n-2)x^{n-3}q_{2n}(1/x)+2(n-2)x^{n-4}q'_{2n}(1/x)-x^{n-5}q''_{2n}(1/x)=0.
	\end{align}
	Finally, taking $x=-1$ and $n=2m$ in \eqref{equ-coeff-q-diff-relation-two} yields \eqref{equ-coeff-q-diff-two}.
\end{proof}

\begin{thm}\label{thm-main-two-sinh}
For any integer $m>0$
\begin{align*}	
\sum_{n=1}^{\infty}\frac{(-1)^nn^{2m+1}}{\sinh(ny)}=\frac{(2m)!}{2^{2m+2}}z^{2m+2}x(x-1)R_{2m}(x)\in\mathbb{Q}[x,z],
\end{align*}
where $R_{2m}(x)=\frac{(x-1)^{m-1}}{(2m)!} q_{2m}\Big(\frac{x}{x-1}\Big) \in \mathbb{Q}[x]$.
\end{thm}
\begin{proof}
From Greenhill's treatise \cite[P. 286, eq.(50)]{Greehill1892},
\begin{align}	
\sn^2 u= \left(1-\frac{E}{K}\right)\frac1{k^2}-\frac{\pi^2}{K^2k^2}\sum_{n=1}^\infty \frac{n \cos(n\pi u/K)}{\sinh(n\pi K'/K)}.
\end{align}
Applying \eqref{rel-den}, the above identity can be rewritten as (also see \cite[eq. (5.8)]{Berndt2016})
\begin{align}	
\sn^2 u= \left(1-\frac{E}{K}\right)\frac1{k^2}-\frac{4}{z^2 x}\sum_{n=1}^\infty \frac{n \cos(2nu/z)}{\sinh(ny)}.
\end{align}
We now apply Jacobi and Ramanujan's process of changing the sign
\[
x\to\frac{x}{x-1},\quad q\to -q,\quad z\to z\sqrt{1-x}.
\]
Then (also see \cite[eq. (5.10)]{Berndt2016}),
\begin{align*} \sn^2\left(u,\frac{k}{i\sqrt{1-k^2}}\right)=
&\left(1-\frac{E}{K}\right)\frac{1}{k^2}-\frac{8(x-1)}{(1-x)xz^2}
\sum_{n=1}^{\infty}\frac{(-1)^nn}{e^{ny}-e^{-ny}}\cos\left(\frac{2nu}{\sqrt{1-x}z}\right)\\
=&\left(1-\frac{E}{K}\right)\frac{1}{k^2}+\sum_{m=0}^{\infty}\frac{8}{z^2x}
\sum_{n=1}^{\infty}\frac{(-1)^nn}{e^{ny}-e^{-ny}}\frac{(-1)^m}{(2m)!}\left(\frac{2nu}{\sqrt{1-x}z}\right)^{2m}\\
=&\left(1-\frac{E}{K}\right)\frac{1}{k^2}+\sum_{m=0}^{\infty}(-1)^m 4\sum_{n=1}^{\infty}\frac{(-1)^n(2n)^{2m+1}}{e^{ny}-e^{-ny}}\frac{u^{2m}}{(2m)!x(1-x)^mz^{2m+2}}.
\end{align*}
On the other hand, from \eqref{equ-one} we can find that $\deg(g_{2n-1}(x))=\deg(q_{2n}(x))=n-1$. Hence, we obtain
\begin{align*}
\sn^2\left(u,\frac{k}{i\sqrt{1-k^2}}\right)=\sum_{m=1}^{\infty}\frac{R_{2m}(x)}{(x-1)^{m-1}}u^{2m}.
\end{align*}
Then, comparing the coefficients of $u^{2m}$ yields the desired result.
\end{proof}

\begin{exa} By \emph{Mathematica}, comparing the coefficients of $u^2,u^4$ and $u^6$, we obtain
	\begin{align*}
		\sum_{n=1}^{\infty}\frac{(-1)^nn^3}{\sinh(ny)}=&\frac{1}{8}z^{4}x(x-1),\\
		\sum_{n=1}^{\infty}\frac{(-1)^nn^5}{\sinh(ny)}=&\frac{1}{8}z^{6}x(x-1)(1-2x),\\
		\sum_{n=1}^{\infty}\frac{(-1)^nn^7}{\sinh(ny)}=&\frac{1}{16}z^{8}x(x-1)(2-17x+17x^2).
	\end{align*}
\end{exa}

Noting that
\begin{align}			
&\sum_{n=1}^{\infty}\frac{(-1)^nn^{2p+3}}{\sinh^3(ny)}=\frac{1}{2}\left(\frac{dx}{dy}\right)^2\frac{d^2}{dx^2}\sum_{n=1}^{\infty}\frac{(-1)^nn^{2p+1}}{\sinh(ny)}
-\frac{1}{2}\sum_{n=1}^{\infty}\frac{(-1)^nn^{2p+3}}{\sinh(ny)}\nonumber\\
&\qquad\qquad\qquad\quad\quad+\frac{1}{2}\frac{d^2y}{dx^2}\left(\frac{dx}{dy}\right)^2\sum_{n=1}^{\infty}\frac{(-1)^nn^{2p+2}
\cosh(ny)}{\sinh^2(ny)}\label{equ-sinh-one},\\
&\sum_{n=1}^{\infty}\frac{(-1)^nn^{2p+4}\cosh(ny)}{\sinh^4(ny)}
=-\frac{1}{3}\frac{dx}{dy}\frac{d}{dx}\sum_{n=1}^{\infty}\frac{(-1)^nn^{2p+3}}{\sinh^3(ny)}\label{equ-sinh-two}
\end{align}
and
\begin{align}		
\sum_{n=1}^{\infty}\frac{(-1)^nn^{2p+5}}{\sinh^5(ny)}=&\frac{1}{3}
\left(\frac{dx}{dy}\right)^2\frac{d^2}{dx^2}\sum_{n=1}^{\infty}\frac{(-1)^nn^{2p+3}}{\sinh^3(ny)}-3\sum_{n=1}^{\infty}\frac{n^{2p+5}\cosh^2(ny)}
{\sinh^5(ny)}\nonumber\\&+\frac{d^2y}{dx^2}\left(\frac{dx}{dy}\right)^2\sum_{n=1}^{\infty}\frac{n^{2p+4}\cosh(ny)}{\sinh^4(ny)}\label{equ-sinh-three},
\end{align}
we have
\begin{align*}
&\sum_{n=1}^{\infty}\frac{(-1)^nn^{5}}{\sinh^3(ny)}=\frac{(x-1)xz^6}{16}\begin{Bmatrix}
			 2x-1+z^2\left(6x^2-6x+1\right)+20(x-1)^2x^2(z')^2\\+2(x-1)xz\big[7(2x-1)z'+2(x-1)xz''\big]
		\end{Bmatrix},\\
		 &\sum_{n=1}^{\infty}\frac{(-1)^nn^{6}\cosh(ny)}{\sinh^4(ny)}\\&=-\frac{(x-1)xz^7}{48}\begin{Bmatrix}
 (24x^3-36x^2+14x-1)z^3+6(x-1)xz'\big(20(z')^2x^2(x-1)^2+2x-1\big)\\
 +z\big(6x^2-6x+1+158(x-1)^2x^2(2x-1)(z')^2+68(x-1)^3x^3z'z''\big)\\
+2(x-1)xz^2\Big[2(47x^2-47x+9)z'
+(x-1)x\big(13(2x-1)z''+2(x-1)xz^{(3)}\big)\Big]
		\end{Bmatrix},\\
 &\sum_{n=1}^{\infty}\frac{(-1)^nn^{7}}{\sinh^5(ny)}\\&=\frac{(x-1)xz^8}{384}\begin{Bmatrix}			 z^4\left(60x(x-1)(2x-1)^2+2\right)\\+3\left[6+51(x-1)x+280(x-1)^2x^2(z')^2\left((z')^2x^2(x-1)^2+2x-1\right)\right]\\
 +8(x-1)xz\Big[10\big(5+27(x-1)x\big)z'+436(x-1)^2x^2(2x-1)(z')^3 \\
 +(x-1)x\big(226(x-1)^2x^2(z')^2+30x-15\big)z''\Big]\\
 +4z^2\Big[5(2x-1)\big(12(x-1)x+1\big)\\
 +(x-1)^2x^2\Big(\big(1952(x-1)x+399\big)(z')^2+34(x-1)^2x^2(z'')^2\\
 +(x-1)xz'\big(411(2x-1)z''+52(x-1)xz^{(3)}\big)\Big)\Big]\\
 +4(x-1)xz^3\Big[(2x-1)\big(342x(x-1)+41\big)z'+2(x-1)^3x^3z^{(4)}\\
 +3(x-1)x\Big(\big(92x(x-1)+19\big)z''+7(x-1)x(2x-1)z^{(3)}\Big) \Big]
\end{Bmatrix}.
\end{align*}

\begin{thm}\label{sinh}
Set $\Gamma=\Gamma(1/4)$ and $R_{n}^{(k)}=R_{n}^{(k)}(1/2)$ for all $n$ and $k$. Then, for any integer $m>1$ we have
\begin{align*}
&\sum_{n=1}^{\infty}\frac{(-1)^nn^{4m-3}}{\sinh^3(n\pi)}
=-\frac{(4m-6)!\Gamma^{8m}}{2^{8m+3}\pi^{6m}} \Big\{4\big(64(m-1)(4m-3)\pi^4/\Gamma^8+m-3\big)R_{4m-6}+R''_{4m-6}\Big\} ,\\	
&\sum_{n=1}^{\infty}\frac{(-1)^nn^{4m-1}}{\sinh^3(n\pi)}
=-\frac{(4m-4)!\Gamma^{8m} }{\pi^{6m+1}2^{8m+3}}\Big\{(4m-1)R'_{4m-4}-2\left(3+2m(4m-5)\right)\pi R_{4m-2}\Big\},\\	
&\sum_{n=1}^{\infty}\frac{(-1)^nn^{4m-2}\cosh(n\pi)}{\sinh^4(n\pi)}
=-\frac{(4m-6)!\Gamma^{8m}}{3\cdot2^{8m+3}\pi^{6m}}\begin{Bmatrix}
256(m-1)(2m-1)(4m-3)\pi^3R_{4m-6}/\Gamma^8\\
-4(m-1)(4m-5)R'_{4m-4}\\
+3(2m-1)\left[4(m-3)R_{4m-6}+R''_{4m-6}\right]/\pi\end{Bmatrix},\\		
&\sum_{n=1}^{\infty}\frac{(-1)^nn^{4m-1}}{\sinh^5(n\pi)}
=-\frac{(4m-1)!\Gamma^{8m-8} R_{4m-6}}{2^{8m+1}\cdot3(4m-5)\pi^{6m-2}}\\
&\quad\quad\quad\quad-\frac{(4m-6)!\Gamma^{8m}}{3\cdot2^{8m+13}\pi^{6m+6}}\begin{Bmatrix}
			256\pi^4\Big[72(m-1)(2m-1)(4m-5)(4m-3)\pi^2 R_{4m-2}  \\
 -40(m-1)(4m-5)(4m-1)\pi R'_{4m-4} \\
 +3(8m^2-6m+1)\big(4(m-3)R_{4m-6}+R''_{4m-6}\big)\Big]\\
+ \Gamma^8 \big[8(6m^2-37m+55)R_{4m-6}
+24(m-4)R''_{4m-6}+R^{(4)}_{4m-6}\big]
		\end{Bmatrix}.
\end{align*}
\end{thm}
\begin{proof}
From Proposition \ref{pro-coeffic-qs}, we get
\begin{align}
	&R_{4m-4}(1/2)=-\frac{1}{2^{2m-3}(4m-4)!}q_{4m-4}(-1)=0,\\
	&R'_{4m-2}(1/2)=-\frac1{2^{2m-4}(4m-2)!} \Big(q'_{4m-2}(-1)+(m-1)q_{4m-2}(-1)\Big)=0,\\
	&R'_{4m-6}(1/2)=-\frac1{2^{2m-6}(4m-6)!} \Big(q'_{4m-6}(-1)+(m-2)q_{4m-6}(-1)\Big)=0,\\
	&R''_{4m-4}(1/2)=-\frac1{2^{2m-7}(4m-4)!} \Big(q''_{4m-4}(-1)+2(m-2)q'_{4m-4}(-1)\Big)=0.
\end{align}
Hence, the theorem follows from Theorem \ref{thm-main-two-sinh}, \eqref{special-case-x} and \eqref{equ-sinh-one}-\eqref{equ-sinh-three}.
\end{proof}

\begin{exa} Set $\Gamma=\Gamma(1/4)$. By \emph{Mathematica}, setting $m=2,3$, we have
	\begin{align*}
		 \sum_{n=1}^{\infty}&\frac{(-1)^nn^{5}}{\sinh^3(n\pi)}=-\frac{5\Gamma^8}{2^{10}\pi^8}+\frac{\Gamma^{16}}{2^{16}\pi^{12}},\\
		 \sum_{n=1}^{\infty}&\frac{(-1)^nn^{7}}{\sinh^3(n\pi)}=\frac{7\Gamma^{16}}{2^{15}\pi^{13}}-\frac{9\Gamma^{16}}{2^{17}\pi^{12}},\\
		 \sum_{n=1}^{\infty}&\frac{(-1)^nn^{9}}{\sinh^3(n\pi)}=\frac{81\Gamma^{16}}{2^{16}\pi^{14}}-\frac{17\Gamma^{24}}{2^{22}\pi^{18}},\\
		 \sum_{n=1}^{\infty}&\frac{(-1)^nn^{11}}{\sinh^3(n\pi)}=-\frac{297\Gamma^{24}}{2^{21}\pi^{19}}+\frac{189\Gamma^{24}}{2^{22}\pi^{18}},\\
		 \sum_{n=1}^{\infty}&\frac{(-1)^nn^{6}\cosh(n\pi)}{\sinh^4(n\pi)}=\frac{3\Gamma^{16}}{2^{16}\pi^{13}}-\frac{\Gamma^{16}}{3\cdot2^{15}\pi^{12}}-\frac{5\Gamma^8}
{2^{10}\pi^9},\\
		 \sum_{n=1}^{\infty}&\frac{(-1)^nn^{10}\cosh(n\pi)}{\sinh^4(n\pi)}=\frac{135\Gamma^{16}}{2^{16}\pi^{15}}-\frac{85\Gamma^{24}}{2^{22}\pi^{19}}
+\frac{9\Gamma^{24}}{2^{21}\pi^{18}},\\
		 \sum_{n=1}^{\infty}&\frac{(-1)^nn^{7}}{\sinh^5(n\pi)}=-\frac{35\Gamma^{8}}{2^{13}\pi^{10}}+\frac{21\Gamma^{16}}{2^{18}\pi^{14}}
-\frac{35\Gamma^{16}}{3\cdot2^{16}\pi^{13}}+\frac{27\Gamma^{16}}{2^{19}\pi^{12}}-\frac{5\Gamma^{24}}{3\cdot2^{25}\pi^{18}},\\
		 \sum_{n=1}^{\infty}&\frac{(-1)^nn^{11}}{\sinh^5(n\pi)}=\frac{1485\Gamma^{16}}{2^{19}\pi^{16}}-\frac{935\Gamma^{24}}{2^{24}\pi^{20}}
+\frac{495\Gamma^{24}}{2^{22}\pi^{19}}-\frac{567 \Gamma^{24}}{2^{24}\pi^{18}}+\frac{65\Gamma^{32}}{2^{31}\pi^{24}}.
	\end{align*}
\end{exa}

\section{Berndt-Type Integrals of Order Three}

\begin{thm}\label{thm-exp-one-BT}  Set $\Gamma=\Gamma(1/4)$ and $p_{n}^{(k)}=p_{n}^{(k)}(1/2)$ for all $n$ and $k$. Then,
for any integer $m>0$ we have
\begin{align}
&\int_{0}^{\infty}\frac{x^{4m+1}\, dx}{(\cos x+\cosh x)^3}\nonumber\\
& =\frac{(-1)^m\Gamma^{8m-4}}{2^{6m+17}\pi^{2m+7}} \begin{Bmatrix}		
256\pi^6\Gamma^8p_{4m+1}-32768m(2m-1)(4m-1)(4m+1)\pi^8p_{4m-3}\\
+512(4m+1)\pi^4\Gamma^8 \Big[\pi p'_{4m-1} + m(4m-6)p_{4m-3}+mp''_{4m-3} \Big]\\
-\Gamma^{16}\left[8(m-2)(6m-7)p_{4m-3}+12(2m-5)p''_{4m-3}+p^{(4)}_{4m-3}\right]
	\end{Bmatrix}.
\end{align}
\end{thm}
\begin{proof}
Using Corollary \ref{corcos+cos} and Theorem \ref{cosh}, we may deduce the desired evaluation after a rather tedious
computation.
\end{proof}

\begin{exa}
 Set $\Gamma=\Gamma(1/4)$. When $m=1,2,3$, we have
\begin{align*}
\int_{0}^{\infty}&\frac{x^5\, dx}{(\cos x+\cosh x)^3}=\frac{15\Gamma^4}{2^8\pi}+\frac{5\Gamma^{12}}{2^{13}\pi^5}-\frac{5\Gamma^{12}}{2^{13}\pi^4}+\frac{3\Gamma^{12}}{2^{15}\pi^3}+\frac{\Gamma^{20}}{2^{20}\pi^9},\\
\int_{0}^{\infty}&\frac{x^9\, dx}{(\cos x+\cosh x)^3}=\frac{567\Gamma^{12}}{2^{13}\pi^3}+\frac{117\Gamma^{20}}{2^{18}\pi^7}-\frac{297\Gamma^{20}}{2^{19}\pi^6}+\frac{189\Gamma^{20}}{2^{21}\pi^5}
+\frac{3\Gamma^{28}}{2^{22}\pi^{11}},\\
\int_{0}^{\infty}&\frac{x^{13}\, dx}{(\cos x+\cosh x)^3}=\frac{405405\Gamma^{20}}{2^{20}\pi^5}+\frac{84591\Gamma^{28}}{2^{25}\pi^9}-\frac{107757\Gamma^{28}}{2^{25}\pi^8}+\frac{68607\Gamma^{28}}
{2^{27}\pi^7}+\frac{17679\Gamma^{36}}{2^{32}\pi^{13}}.
\end{align*}
\end{exa}

By numerical computation we believe the following evaluation holds.

\begin{conj} We have
\begin{equation*}
\int_{0}^{\infty} \frac{x\, dx}{(\cos x+\cosh x)^3}=-\frac{\Gamma^4}{2^7\pi^2}+\frac{\Gamma^4}{2^9\pi}+\frac{\Gamma^{12}}{2^{13}\pi^7}.
\end{equation*}
\end{conj}
To prove this conjecture using our method in this paper, we are required to study the series
\begin{equation*}
\sum_{n=1}^{\infty}\frac{(-1)^{n}}{(2n-1)\cosh^3(\tilde{n}y)}
\end{equation*}
by Cor.~\ref{corcos+cos}. It seems that we need new ideas to evaluate this precisely.

\begin{thm}\label{thm-exp-two-BT}
Set $\Gamma=\Gamma(1/4)$ and $R_{n}^{(k)}=R_{n}^{(k)}(1/2)$ for all $n$ and $k$. Then, for any integer $m>1$ we have
\begin{align}
&\int_{0}^{\infty}\frac{x^{4m-1}\, dx}{(\cos x-\cosh x)^3}=\nonumber\\
&\frac{(-1)^m\Gamma^{8m}}{2^{6m+7}\pi^{2m+2}}\begin{Bmatrix}
	2\pi\left[(4m-4)!(4m-1)R'_{4m-4}-\frac{8\pi^2 (4m-1)!R_{4m-6}}{(4m-5)\Gamma^8}\right]\\
+(4m-6)!\Big[(2m-1)\Big(8(m-1)(4m-5)(4m-3)\pi^2R_{4m-2} \\
+(4m-1)\big(4(m-3)R_{4m-6}+R''_{4m-6}\big)\Big)\\
-\frac{\Gamma^8}{256\pi^4}\Big(8(55+m(6m-37))R_{4m-6}+24(m-4)R''_{4m-6}+R^{(4)}_{4m-6}\Big)\Big]
\end{Bmatrix}.
\end{align}
\end{thm}
\begin{proof}
Using Corollary \ref{corcos-cosh} and Theorem \ref{sinh}, we may deduce the desired evaluation after a rather tedious
computation.
\end{proof}

\begin{exa}
Set $\Gamma=\Gamma(1/4)$. When $m=2,3$, we have
\begin{align*}
\int_{0}^{\infty}&\frac{x^7\, dx}{(\cos x-\cosh x)^3}
=-\frac{105\Gamma^8}{2^{11}\pi^2}-\frac{21\Gamma^{16}}{2^{16}\pi^6}+\frac{7\Gamma^{16}}{2^{14}\pi^5}-\frac{9\Gamma^{16}}{2^{17}\pi^4}
-\frac{5\Gamma^{24}}{2^{23}\pi^{10}},\\
\int_{0}^{\infty}&\frac{x^{11}\, dx}{(\cos x-\cosh x)^3}=-\frac{4455\Gamma^{16}}{2^{15}\pi^4}-\frac{935\Gamma^{24}}{2^{20}\pi^8}+\frac{297\Gamma^{24}}{2^{18}\pi^7}
-\frac{189\Gamma^{24}}{2^{20}\pi^6}-\frac{195\Gamma^{32}}{2^{27}\pi^{12}}.
\end{align*}
\end{exa}

Hence, from Theorems \ref{thm-exp-one-BT} and \ref{thm-exp-two-BT} we obtain Theorem \ref{Main-one-theorem}.

\medskip
{\bf Acknowledgments.} Ce Xu is supported by the National Natural Science Foundation of China (Grant No. 12101008), the Natural Science Foundation of Anhui Province (Grant No. 2108085QA01) and the University Natural Science Research Project of Anhui Province (Grant No. KJ2020A0057). Jianqiang Zhao is supported by the Jacobs Prize from The Bishop's School.


\begin{thebibliography}{99}

\bibitem{A2000}
G.E. Andrews, R. Askey and R. Roy, \emph{Special Functions}, Cambridge University Press, 2000.

\bibitem{B1991} B.C. Berndt, \emph{Ramanujan's Notebooks, Part III}, Springer-Verlag, New York, 1990,\ pp. 87--142.

\bibitem{Berndt2016}
B.C. Berndt, Integrals associated with Ramanujan and elliptic functions, \emph{Ramanujan J.}, {\bf 41}(2016),\ pp. 369--389.

\bibitem{BB2002}
B.C.\ Berndt, P.R.\ Bialek and A.J.\ Yee, Formulas of Ramanujan for the power series coefficients of certain quotients of Eisenstein series, \emph{Int. Math. Res. Notices}
\textbf{21}(2002), pp.\ 1077-1109.

\bibitem{Campbell}
J.M. Campbell, Hyperbolic summations derived using the Jacobi functions ${\rm dc}$ and ${\rm nc}$, arxiv: 2301.03738.

\bibitem{DCLMRT1992}
G. Dattoli, C. Chiccoli, S. Lorenzutta, G. Maino, M. Richetta and A. Torre, Generating
functions of multivariable generalized Bessel functions and Jacobi-elliptic functions, \emph{J.
Math. Phys.}, {\bf 33}(1992),\ pp.\ 25--36.

\bibitem{Greehill1892}
A.G. Greenhill, The applications of elliptic functions, Macmillan, London (1992).

\bibitem{Hancock1910}
H. Hancock, \emph{Lectures on the Theory of Elliptic Functions}, 1st ed.,
John Wiley and Sons, New York, 1910.

\bibitem{Ismail1998}
M.E.H. Ismail and G. Valent, On a family of orthogonal polynomials related to elliptic functions, \emph{Illinois J. Math.}, \textbf{42}(1998), pp.\ 294--312.

\bibitem{T2015}
Y.\ Komori, K.\ Matsumoto and H.\ Tsumura, Infinite series involving hyperbolic functions, \emph{Lith. Math. J.}, \textbf{55}(2015), pp.\ 102--118.

\bibitem{K2017}
A.\ Kuznetsov, A direct evaluation of an integral of Ismail and Valent, arXiv:1607.08001.

\bibitem{Rama1916}
S.\ Ramanujan, On certain arithmetic functions, \emph{Trans. Cambridge Phil.\ Soc.}, \textbf{22} (9) (1916), pp.\ 159-184. Also available in \emph{Collected Papers of Srinivasa Ramanujan}, ed.
G.H.\ Hardy, P.V.\ Seshu Aiyar and B.M.\ Wilson, Cambridge University Press, 1927, 134-147.

\bibitem{T2008}
H.\ Tsumura, On certain analogues of Eisenstein series and their evaluation formulas of Hurwitz type, \emph{Bull.\ London Math.\ Soc.}, \textbf{40}(2008), pp.\ 289--297.

\bibitem{T2010}
H.\ Tsumura, Analogues of the Hurwitz formulas for level 2 Eisenstein series, \emph{Results in Math.}, \textbf{58}(2010), pp.\ 365--378.

\bibitem{T2012}
H.\ Tsumura, Analogues of level-$N$ Eisenstein series, \emph{Pacific J.\ Math.}, \textbf{265}(2012), pp.\ 489--510.

\bibitem{W1916}
A.C.L. Wilkinson, Solution to problem 353, \emph{Indian J. Pure. Appl. Math.}, \textbf{8}(1916), pp.\ 106--110.

\bibitem{WW1966}
E.T.\ Whittaker and G.N.\ Watson, \emph{A course of modern analysis, 4th edn}, Cambridge University Press, Cambridge, 1966, 511-512.

\bibitem{X2018}
C. Xu, Some evaluation of infinite series involving trigonometric and hyperbolic functions, \emph{Results. Math.}, \textbf{73}(2018), article number: 128.

\bibitem{XuZhao-2022}
C. Xu and J. Zhao, Functional relations for Ramanujan type sums of reciprocal hyperbolic functions. arXiv:1801.07565v4.

\bibitem{XZ2023}
C. Xu and J. Zhao, Berndt-type integrals and series associated with ramanujan and Jacobi elliptic functions. arXiv:2301.08211.

\bibitem{Ya-2018}
S. Yakubovich, On the curious series related to the elliptic integrals, \emph{Ramanujan J.}, {\bf 45}(2018),\ pp. 797--815.

\bibitem{GR}
D.\ Zwillinger, Table of Integrals, Series, and Products, 8th ed., Academic Press, 2015.

\end{thebibliography}
\end{document}